\documentclass[12pt]{article}

\usepackage{amsmath, amsthm, amssymb, amsfonts, graphicx, mathtools, color, setspace, fancyhdr}
\usepackage[margin=1in]{geometry}
\usepackage[utf8]{inputenc}
\usepackage[english]{babel}
\usepackage[svgnames]{xcolor}
\usepackage{tikz}
\usetikzlibrary{positioning, patterns}
\usepackage{setspace}
\usepackage[numbers]{natbib}
\usepackage[linkcolor=Blue,colorlinks=true,citecolor=ForestGreen,filecolor=red,urlcolor=Blue]{hyperref}
\usepackage{cleveref}[2011/12/24]
\usepackage{authblk}

\newtheorem{theorem}{Theorem}
\newtheorem{corollary}[theorem]{Corollary}
\newtheorem{lemma}[theorem]{Lemma}

\newtheorem{question}[theorem]{Question}
\newtheorem{conjecture}[theorem]{Conjecture}

\theoremstyle{definition}
\newtheorem{example}[theorem]{Example}
\newtheorem{remark}[theorem]{Remark}

\newcommand{\rank}{\mathrm{rank}}

\newcommand{\G}{\mathcal{G}}

\renewcommand{\Cup}{\bigcup}

\renewcommand{\setminus}{-}

\newcommand{\st}{\colon\,}
\newcommand{\gal}{\mathrm{lpt}}
\newcommand{\lpt}{\mathrm{lpt}}

\newcommand{\abs}[1]{| #1 |}
\newcommand{\dist}{\mathrm{dist}}
\newcommand{\set}[1]{\left \{#1 \right\} }

\renewcommand{\emptyset}{\varnothing}

\newcommand{\comment}[1]{}

\newtheorem{proposition}[theorem]{Proposition}

\begin{document}

\author[1]{James A. Long Jr.\thanks{jalong@mix.wvu.edu}}
\author[1]{Kevin G. Milans\thanks{milans@math.wvu.edu}}
\author[2]{Andrea Munaro\thanks{andrea.munaro@unipr.it}}
\affil[1]{Department of Mathematics, West Virginia University, USA}
\affil[2]{Department of Mathematical, Physical and Computer Sciences, University of Parma, Italy}

\title{Non-empty intersection of longest paths in $H$-free graphs}
\date{\today}

\maketitle

\begin{abstract}We make progress toward a characterization of the graphs $H$ such that every connected $H$-free graph has a longest path transversal of size $1$. In particular, we show that the graphs $H$ on at most $4$ vertices satisfying this property are exactly the linear forests. We also show that if the order of a connected graph $G$ is large relative to its connectivity $\kappa(G)$, and its independence number $\alpha(G)$ satisfies $\alpha(G) \le \kappa(G) + 2$, then each vertex of maximum degree forms a longest path transversal of size $1$.
\end{abstract}

\section{Introduction}\label{intro}

It is a classic result in graph theory that every two longest paths in a connected graph share at least one vertex. \citet{Gal68} asked whether in fact all longest paths in a connected graph share at least one vertex. This was answered in the negative by \citet{Wal69}, who provided a counterexample with $25$ vertices. A counterexample with $12$ vertices was later constructed by \citet{WV74} and, independently, by \citet{Zam76} (see \Cref{fig:counter}). Brinkmann and Van Cleemput (see \citep{SZZ13}) verified that there is no counterexample with less than $12$ vertices. 

\begin{figure}[h!]
\centering 
\includegraphics[scale=0.8]{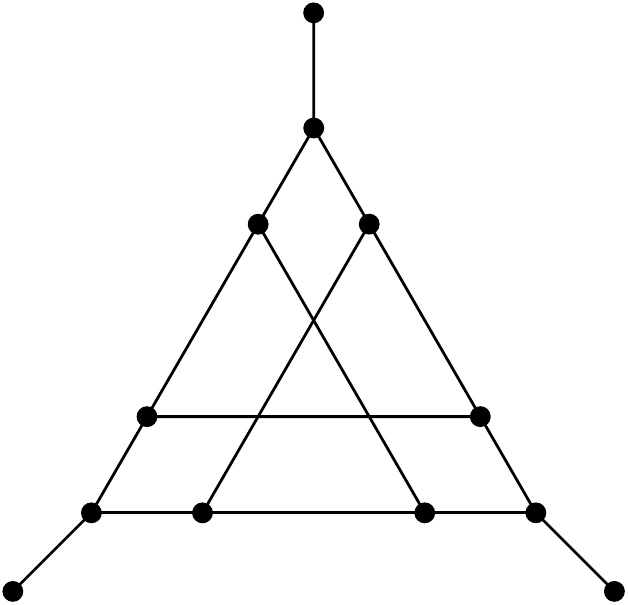}
\caption{The graph $G_0$: A $12$-vertex graph with no Gallai vertex.}
\label{fig:counter}
\end{figure}

A \emph{Gallai set} (or \emph{longest path transversal}) in a graph $G$ is a set of vertices $S$ such that every longest path in $G$ has a vertex in $S$.  The \emph{Gallai number} or \emph{longest path transversal number} of $G$, denoted by $\gal(G)$, is the minimum size of a Gallai set and a \emph{Gallai family} is a family of graphs $\G$ such that $\gal(G)=1$ for each connected graph $G\in \G$.  A vertex $v$ in $G$ is a \emph{Gallai vertex} if $\{v\}$ is a Gallai set and a graph is \emph{Gallai} if it has a Gallai vertex.

The counterexamples mentioned above consist of connected graphs $G$ for which \linebreak $\lpt(G) = 2$. In fact, there are examples of connected graphs $G$ for which $\lpt(G) = 3$ \citep{Gru74,Zam76} and \Citet{Wal69} and \citet{Zam72} asked if the Gallai number of connected graphs is bounded. In a companion paper \cite{LMM21} we addressed this fifty-year-old question. Improving on \cite{RS14}, we showed that connected graphs admit sublinear longest path transversals. The gap between our upper bound and the constant lower bound $3$ remains large. 

In this paper we focus on another natural variant of Gallai's question: Which classes of graphs form Gallai families? It is well known that a family of pairwise intersecting subtrees of a tree has non-empty intersection; in particular, trees form a Gallai family.  Several other Gallai families have been identified: split graphs and cacti \citep{KP90}, circular-arc graphs \citep{BGLS04,Joos15}, series-parallel graphs \citep{CEF17},  graphs with matching number at most $3$ \citep{Chen15}, dually chordal graphs \citep{JKLW16}, $2K_2$-free graphs \citep{GS18}, $P_4$-sparse graphs and $(P_5, K_{1,3})$-free graphs \citep{CL20}, bipartite permutation graphs \citep{CFG20}, $(H_1, H_2)$-free graphs such that $H_1$ and $H_2$ are connected and every $2$-connected $(H_1, H_2)$-free graph is Hamiltonian (all such pairs are known and each includes $K_{1,3}$) \citep{GS19}.

\newcommand{\Free}{\mathrm{Free}}
Let $\Free(H)$ be the class of $H$-free graphs.  A \emph{monogenic class} of graphs has the form $\Free(H)$, for some graph $H$. In this paper we aim at characterizing monogenic Gallai families. In \Cref{monogenic}, we make progress by showing that if $\Free(H)$ is a Gallai family, then $H$ is a linear forest, and this suffices when $|V(H)|\le 4$. In the spirit of \citep{GS18}, we in fact prove something more general: if $H$ is a linear forest on at most $4$ vertices and $G$ is a connected $H$-free graph, then all maximum degree vertices in $G$ are Gallai. Dichotomies in monogenic classes for structural and algorithmic graph properties have been the subject of several studies. For example, they have been provided for properties such as Hamiltonicity \citep{FG97,LV17}, boundedness of clique-width \citep{DP16}, price of connectivity \citep{BHKP17,HJMP16}, and polynomial-time solvability of various algorithmic problems \citep{GPS14,Kor92,KKTW01,Mun17}. In \Cref{sec:5P1}, we show that if $G$ is a connected graph with independence number $\alpha(G)\le 4$ (i.e., $G$ is $5P_1$-free), then $G$ is Gallai. We then conjecture that the same holds if $\alpha(G) \leq 5$.

A celebrated result of \Citet{CE72} asserts that a graph $G$ has a Hamiltonian cycle when $|V(G)| \ge 3$ and $\alpha(G) \le \kappa(G)$, and that $G$ has a Hamiltonian path when $\alpha(G) \le \kappa(G)+1$.  It follows that every vertex in $G$ is Gallai when $\alpha(G)\le \kappa(G) + 1$. In \Cref{sec:CELR}, we show that if a connected graph $G$ is large relative to its connectivity $\kappa(G)$ and $\alpha(G) \le \kappa(G) + 2$, then each vertex of maximum degree is a Gallai vertex.  Moreover, for each $k\ge 1$, we provide an infinite family of $k$-connected graphs $G$ such that $\alpha(G)\le k+3$ but no maximum degree vertex in $G$ is Gallai (see \Cref{Ex:BestPossible}). Our result has the following immediate consequence: if a regular graph $G$ is large relative to its connectivity and $\alpha(G) \leq \kappa(G)+2$, then $G$ contains a Hamiltonian path. The condition $\alpha(G) \leq \kappa(G)+2$ is best possible up to an additive factor of $2$ (this follows from a construction in  \citep{CO13}, see \Cref{Ex:HamReg}).

\section{Preliminaries}\label{chap:prel}
In this paper we consider only finite graphs. Given a graph $G$, we denote its vertex set by $V(G)$ and its edge set by $E(G)$. 

\vspace{0.3cm}

\textbf{Neighborhoods and degrees.} For a vertex $v \in V(G)$, the \textit{neighborhood} $N_{G}(v)$ is the set of vertices adjacent to $v$ in $G$.  For a set of vertices $S\subseteq V(G)$, the \textit{neighborhood} of $S$, denoted $N_G(S)$, is $\Cup_{v\in S} N_G(v)$.  We also extend the concept of neighborhood to subgraphs by defining $N_G(H)=N_G(V(H))$ when $H$ is a subgraph of $G$.  The \textit{degree} $d_{G}(v)$ of a vertex $v \in V(G)$ is the number of edges incident to $v$ in $G$.  When $G$ is clear from context, we may write $d(v)$ for $d_G(v)$.  A vertex $v\in V(G)$ with $d(v)=3$ is \emph{cubic}. The \textit{maximum degree} $\Delta(G)$ of $G$ is  $\max\left\{d_{G}(v): v \in V\right\}$. Similarly, the \textit{minimum degree} $\delta(G)$ of $G$ is $\min\left\{d_{G}(v): v \in V\right\}$. 

\vspace{0.3cm}

\textbf{Paths and cycles.}  A \textit{path} is a non-empty graph $P$ with $V(P) = \left\{x_{0}, x_{1}, \dots, x_{k}\right\}$ and $E(P) = \left\{x_{0}x_{1}, x_{1}x_{2}, \dots, x_{k-1}x_{k}\right\}$.  We may also denote $P$ by listing its vertices in the natural order $x_0x_1\cdots x_k$.  The vertices $x_{0}$ and $x_{k}$ are the \textit{ends} or \emph{endpoints} of $P$; the other vertices are \textit{interior} vertices of $P$. The \textit{length} of $P$ is the number of edges in $P$.  We denote the $n$-vertex path by $P_n$. A path in a graph $G$ is \textit{Hamiltonian}, or \textit{spanning}, if it contains all vertices of $G$. A $uv$-path is a path whose endpoints are $u$ and $v$. 
If $P = x_0x_1\cdots x_k$ is a path and $k \geq 2$, the graph with vertex set $V(P)$ and edge set $E(P) \cup {x_k x_0}$ is a \textit{cycle}. The \emph{length} of a cycle is the number of its edges (or vertices) and the cycle on $n$ vertices is denoted by $C_{n}$. A cycle in a graph $G$ is \textit{Hamiltonian}, or \textit{spanning}, if it contains all vertices of $G$.  The \textit{girth} of a graph containing a cycle is the length of a shortest cycle and a graph with no cycle has infinite girth. The \emph{distance} $\dist_G(u, v)$ from a vertex $u$ to a vertex $v$ in a graph $G$ is the length of a shortest path between $u$ and $v$. 

\vspace{0.3cm}

\textbf{Graph operations.} Let $G$ be a graph and let $S\subseteq V(G)$.  The graph $G-S$ is obtained from $G$ by deleting all vertices in $S$ and all edges incident to a vertex in $S$.  The subgraph of $G$ induced by a set of vertices $S'$, denoted $G[S']$, is the graph $G-S$, where $S=V(G)-S'$.  For $M\subseteq E(G)$, we define $G-M$ analogously. The \textit{union} of simple graphs $G$ and $H$ is denoted $G \cup H$ and has vertex set $V(G) \cup V(H)$ and edge set $E(G) \cup E(H)$.  The \emph{disjoint union} of $G$ and $H$, denoted $G+H$, is the union of a copy of $G$ and a copy of $H$ on disjoint vertex sets.  The disjoint union of $k$ copies of $G$ is denoted by $kG$.  
     
\vspace{0.3cm}  

\textbf{Graph classes and special graphs.} If a graph does not contain induced subgraphs isomorphic to graphs in a set $Z$, it is \textit{$Z$-free} and the set of all $Z$-free graphs is denoted by $\Free(Z)$. A \textit{complete graph} is a graph whose vertices are pairwise adjacent and the complete graph on $n$ vertices is denoted by $K_{n}$.  A \textit{triangle} is the graph $K_{3}$. A graph $G$ is \textit{$r$-partite}, for $r \geq 2$, if its vertex set admits a partition into $r$ classes such that every edge has its endpoints in different classes. An $r$-partite graph in which every two vertices from distinct parts are adjacent is called \textit{complete} and $2$-partite graphs are usually called \textit{bipartite}. An \textit{$(X,Y)$-bigraph} is a bipartite graph with bipartition $\set{X, Y}$. Given a graph $G$ and $X, Y \subseteq V(G)$, the \textit{induced $(X, Y)$-bigraph} is the bipartite subgraph of $G$ with vertex set $X \cup Y$ and where each edge has one endpoint in $X$ and the other in $Y$.  A \textit{tree} is a connected graph not containing any cycle as a subgraph and the vertices of degree $1$ are its \textit{leaves}. 
           
\vspace{0.3cm}           
           
\textbf{Graph parameters.} A set of vertices or edges of a graph is \textit{maximum} with respect to the property $\mathcal{P}$ if it has maximum size among all subsets having property $\mathcal{P}$. An \textit{independent set} of a graph is a set of pairwise non-adjacent vertices and the \textit{independence number} $\alpha(G)$ is the size of a maximum independent set of $G$. A \textit{clique} of a graph is a set of pairwise adjacent vertices. A \emph{matching} in $G$ is a set of edges with distinct endpoints.  A matching $M$ \emph{saturates} a set of vertices $S$ if each vertex in $S$ is the endpoint of an edge in $M$.  A graph $G$ is \emph{$k$-connected} if $|V(G)| > k$ and $G-S$ is connected for each $S\subseteq V(G)$ with $|S|<k$.   The \emph{connectivity} of $G$, denoted $\kappa(G)$, is the maximum $k$ such that $G$ is $k$-connected.

\newcommand{\vep}{\varepsilon}
\newcommand{\mst}{\tau}

\section{Monogenic Gallai families}\label{monogenic}

In this section we make progress toward a classification of monogenic Gallai families. We first show that a necessary condition for a monogenic family $\Free(H)$ to be Gallai is that $H$ is a linear forest on at most $9$ vertices, where a \emph{linear forest} is a forest in which every component is a path.  Let $G_0$ be the graph in \Cref{fig:counter} with $\lpt(G_0)=2$ \citep{WV74,Zam76}. We obtain necessary conditions on monogenic Gallai families by subdividing edges or replacing cubic vertices with triangles in $G_0$ to obtain new counterexamples with arbitrarily large girth or no induced claw, respectively. 

In the following, we say that a graph $H$ is a \emph{fixer} if $\Free(H)$ is a Gallai family; that is, forbidding $H$ ``fixes'' the answer to Gallai's question.  

\begin{proposition}\label{prop:linforest}
If $H$ is a fixer, then $H$ is a linear forest on at most 9 vertices.
\end{proposition}
\begin{proof}
Let $H$ be a fixer. By definition, if $G$ is a graph with $\lpt(G)>1$, then $H$ is an induced subgraph of $G$. 

Note that $G_0$ is obtained from the Petersen graph by splitting an arbitrary vertex into a set $R$ of three vertices, each of degree 1 (see \Cref{fig:counter}). Clearly, $G_0$ is triangle-free and every path in $G_0$ avoids at least one vertex in $R$.  Since the Petersen graph has no Hamiltonian cycle \citep{West}, every path in $G_0$ omits at least $2$ vertices. Moreover, since the Petersen graph is vertex-transitive \citep{West} and has a $9$-cycle, it follows that for each vertex $x\in V(G_0)\setminus R$, there is a longest path in $G_0$ with both ends in $R$ that omits only $x$ and the other vertex in $R$.

Let $M$ be the set of $3$ edges incident to the vertices in $R$.  Let $G_1$ be the graph obtained from $G_0$ by replacing each edge in $M$ with a path of length $q$ and replacing each edge outside $M$ with a path of length $p$, where $p>|V(H)|$.  Provided that $q>|E(G_0)|\cdot p$, the longest paths in $G_1$ are in bijective correspondence with the longest paths in $G_0$ that have both ends in $R$.  Recalling that, for each $x\in V(G_0)\setminus R$, there is a longest path in $G_0$ with both ends in $R$ that omits $x$, we have $\gal(G_1)>1$.  Since $G_1$ has girth larger than $|V(H)|$ and $H$ is an induced subgraph of $G_1$, it follows that $H$ is acyclic.  

Let $S$ be the set of cubic vertices in $G_1$. We obtain $G_2$ from $G_1$ by replacing each vertex $w\in S$ with a triangle $T_w$ such that the three edges incident to $w$ in $G_1$ are incident to distinct vertices of $T_w$ in $G_2$. Clearly, $G_2$ is claw-free. Let $P$ be a longest path in $G_2$. Again, provided that $q$ is sufficiently large, $P$ has its ends in $R$. When $P$ visits a vertex in some $T_w$, it must visit all vertices in $T_w$ before leaving. It follows that the longest paths in $G_2$ are in bijective correspondence with the longest paths in $G_1$ and $\gal(G_2)>1$.

Since $H$ is an induced subgraph of $G_1$ and $G_2$, it follows that $H$ is triangle-free and claw-free, and so $\Delta(H)\le 2$.  Recalling that $H$ is acyclic, we have that $H$ is a linear forest. But $H$ is also an induced subgraph of $G_0$ and to obtain an induced linear forest as a subgraph of $G_0$, a vertex must be deleted from the closed neighborhood of each cubic vertex of $G_0$.  Let $R'$ be the set of neighbors of vertices in $R$.  Since the vertices in $R'$ are cubic and have disjoint closed neighborhoods, each induced linear forest has at most $|V(G_0)| - |R'|$ vertices, and so  $|V(H)|\le |V(G_0)| - |R'| = 12 - 3 = 9$.  
\end{proof}

\begin{remark}
\citet[Problem~6]{GS19} asked whether all longest paths in a connected claw-free graph have a non-empty intersection. \Cref{prop:linforest} answers this question in the negative.
\end{remark}

For $|V(H)| \le 4$, we show that $H$ is a fixer if and only if $H$ is a linear forest.  Necessity follows from \Cref{prop:linforest}. For sufficiency, we show that every $4$-vertex linear forest is a fixer. The linear forests of order $4$ are $P_4$, $P_3 + P_1$, $2P_2$, $P_2 + 2P_1$, and $4P_1$ (see \Cref{fig:4vertex}). \citet{CL20} showed that $P_4$-sparse graphs, a superclass of $P_4$-free graphs, form a Gallai family, whereas \citet{GS18} showed that $2P_2$-free graphs form a Gallai family. In other words, $P_4$ and $2P_2$ are fixers. In the following, we address the remaining cases: $P_3 + P_1$, $P_2 + 2P_1$, and $4P_1$.

\begin{figure}[h!]
\centering
\includegraphics[scale=0.9]{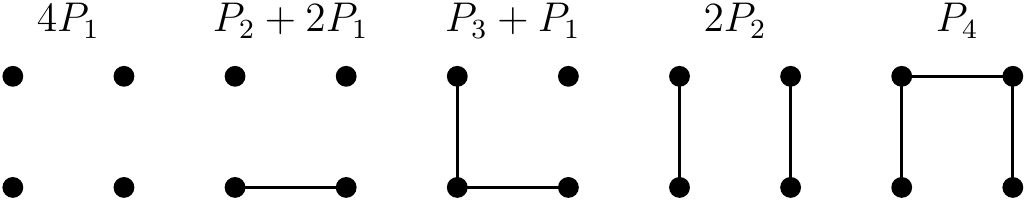}
\caption{The linear forests on $4$ vertices. These are exactly the graphs $H$ on $4$ vertices such that $\Free(H)$ is a Gallai family.}
\label{fig:4vertex}
\end{figure}

We begin with some basic but useful observations.  Given vertices $x,y\in V(G)$, an \emph{$xy$-fiber} is a longest path among all the $xy$-paths. Similarly, an \emph{$x$-fiber} is a longest path among all the paths having $x$ as an endpoint, and a \emph{fiber} is a longest path in $G$.  Note that every fiber is an $x$-fiber for some vertex $x$, and every $x$-fiber is an $xy$-fiber for some vertex $y$.

The following two basic lemmas are used repeatedly, sometimes implicitly. Similar ideas are key to the results in \citep{CE72}. The first basic lemma treats single neighbors of fibers.

\begin{lemma}\label{lem:x-fiber}
Let $P$ be an $xy$-path in a graph $G$, where $P=v_0\cdots v_\ell$ with $x=v_0$ and $y=v_\ell$.  Let $H$ be a component of $G-V(P)$ with a neighbor $v_i$ on $P$.  If $P$ is an $x$-fiber, then $i<\ell$.  Moreover, if $0<i$, then $v_\ell v_{i-1}\not\in E(G)$.  Similarly, if $P$ is a $y$-fiber, then $0<i$, and if $i<\ell$, then $v_0v_{i+1}\not\in E(G)$.
\end{lemma}
\begin{proof}
Suppose $P$ is an $x$-fiber.  No vertex in $H$ is adjacent to $y$, or else $P$ extends to a longer $x$-fiber, a contradiction.  Therefore, $i<\ell$.  Also, if $i>0$ and $v_{i-1}v_\ell \in E(G)$, then following $P$ from $v_0$ to $v_{i-1}$, traversing $v_{i-1}v_\ell$, following $P$ backward from $v_\ell$ to $v_i$, and traveling to $H$ produces a longer $x$-fiber.  The case that $P$ is a $y$-fiber is symmetric.
\end{proof}

In many of our arguments, we show that a path $P$ in $G$ has some desired property or else we obtain a longer path.  We now formalize two common ways to obtain longer paths.  Given two lists of objects $a$ and $b$, a \emph{splice} of $a$ with $b$ is a sequence obtained from $a$ by (1) replacing a non-empty interval of $a$ with $b$, or (2) inserting $b$ between consecutive elements in $a$, or (3) prepending or appending $b$ to $a$.  Given a \emph{host path} $P$ and a \emph{patching path} $Q$, a \emph{splice} of $P$ with $Q$ is a path whose vertices are ordered according to a splice of the ordered list of vertices in $P$ with the ordered list of vertices in $Q$.  A splice of $P$ that has the same endpoints as $P$ is an \emph{interior splice}; otherwise, the splice is \emph{exterior}.

A \emph{detour} of an $xy$-path $P$ is a path obtained from $P$ by using two patching paths $Q_1$ and $Q_2$ as follows.  Suppose that $Q_i$ is a $u_iw_i$-path for $i\in\{1,2\}$ and $u_1,u_2,w_1,w_2$ are distinct vertices appearing in order along $P$.  We follow $P$ from $x$ to $u_1$, traverse $Q_1$, follow $P$ backward from $w_1$ to $u_2$, traverse $Q_2$, and finally follow $P$ from $w_2$ to $y$.

Note that our definitions of a splice and detour require the resulting object to be a path and therefore implicitly impose certain disjointness conditions on segments of the host and the patching paths. Also, note that interior splices and detours of $P$ have the same endpoints as $P$. A splice or detour of $P$ is \emph{augmenting} if it is longer than $P$.

Let $P$ be a path in $G$ and let $H$ be a component of $G-V(P)$. A vertex $s\in V(P)$ with a neighbor in $H$ is an \emph{attachment point} of $H$.  Our next lemma concerns pairs of attachment points. 

\begin{lemma}\label{lem:attach-pair}
Let $P$ be an $xy$-path in a graph $G$ and let $H$ be a component of $G-V(P)$ with attachment points $s$ and $s'$, where $s$ appears before $s'$ when traversing $P$ from $x$ to $y$. The following hold:
\begin{enumerate}
	\item If $s$ and $s'$ are consecutive on $P$, then there is an augmenting interior splice of $P$.
	\item If $s$ and $s'$ are not consecutive along $P$, $w$ and $w'$ immediately follow $s$ and $s'$ respectively, and $ww'\in E(G)$, then there is an augmenting detour of $P$.
	\item If $s$ and $s'$ are not consecutive along $P$, $w$ and $w'$ immediately precede $s$ and $s'$ respectively, and $ww' \in E(G)$, then there is an augmenting detour of $P$.
\end{enumerate}
\end{lemma}
\begin{proof}
For part 1, since $s$ and $s'$ are consecutive attachment points on $P$, we obtain an augmenting interior splice by inserting an appropriate path in $H$ between $s$ and $s'$.  For part 2, let $Q_1$ be an $ss'$-path with interior vertices in $H$ and let $Q_2$ be the path $ww'$.  There is an augmenting detour of $P$ using patching paths $Q_1$ and $Q_2$. The case in part 3 is symmetric.
\end{proof}

When $P$ is a kind of fiber and a component $H$ of $G-V(P)$ has many attachment points, our next lemma obtains a large independent set contained in $P$ consisting of non-attachment points.

\begin{lemma}\label{lem:indep-set}
Let $P$ be an $xy$-path in a graph $G$, let $H$ be a component of $G-V(P)$ and let $k$ be the number of attachment points of $H$. There is an independent set $A$ of $G$ such that $A \subseteq V(P)$, no edge joins a vertex in $A$ and a vertex in $V(H)$, and the following hold:
\begin{enumerate}
    \item If $P$ is an $xy$-fiber, then $A \subseteq V(P)-\set{x,y}$ and $|A|\ge k-1$.
    \item If $P$ is an $x$-fiber, then $A \subseteq V(P)-\set{x}$ and $|A|\ge k$.
    \item If $P$ is a fiber, then $A \subseteq V(P)$ and $|A|\ge k+1$.
\end{enumerate}
\end{lemma}
\begin{proof}
Let $s_1,\ldots,s_k$ be the attachment points of $H$, with indices increasing from $x$ to $y$ along $P$, and let $S = \{s_1,\ldots,s_k\}$ (see \Cref{fig:lemma5}).  

For part 1, let $A$ be the set of vertices in $P$ that immediately follow some $s_i$ with $1\le i < k$. Since $P$ is an $xy$-fiber, \Cref{lem:attach-pair} implies that $s_i$ and $s_{i+1}$ are not consecutive along $P$. Therefore, $S$ and $A$ are disjoint and so no vertex in $A$ has a neighbor in $H$. By \Cref{lem:attach-pair}, it follows that $A$ is an independent set.  

For part 2, suppose in addition that $P$ is an $x$-fiber.  By \Cref{lem:x-fiber}, $s_k \ne y$, and we may take $A$ to be the set of vertices that immediately follow some $s_i$ with $1\le i\le k$.  

For part 3, suppose in addition that $P$ is a fiber.  By \Cref{lem:x-fiber}, we have $s_1 \ne x$.  Let $A$ be the set of vertices that immediately follow an attachment point together with $x$.  Note that since $P$ is also a $y$-fiber, it follows from \Cref{lem:x-fiber} that $x$ has no neighbor in $A$, and so $A$ is an independent set of size $k+1$.  
\end{proof}

\begin{figure}[h!]
\centering 
\includegraphics[scale=0.9]{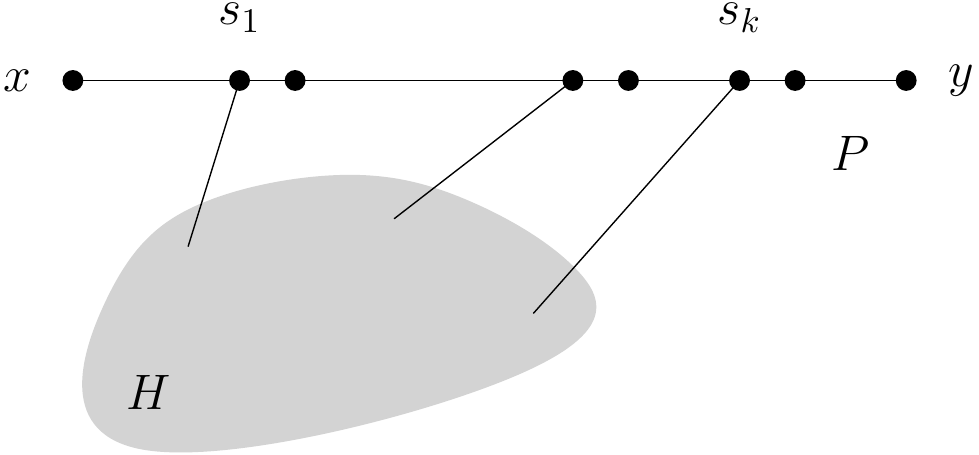}
\caption{Construction of $A$ in the proof of \Cref{lem:indep-set}.}
\label{fig:lemma5}
\end{figure}

We can finally show in the following sections that $P_3 + P_1$, $P_2 + 2P_1$, and $4P_1$ are all fixers. 

\subsection{$P_3 + P_1$ is a fixer}

\begin{theorem}\label{thm:P3+P1}
If $G$ is a connected $(P_3+P_1)$-free graph, then every vertex of degree at least $\Delta(G) - 1$ is a Gallai vertex.
\end{theorem}
\begin{proof}
Let $P$ be a longest path in $G$, where $P=v_0\cdots v_{\ell}$ with $x=v_0$ and $y=v_\ell$. Suppose for a contradiction that there is a vertex $u$ with $d(u)\ge\Delta(G) - 1$ but $u\not\in V(P)$.  Let $H$ be the component of $G-V(P)$ containing $u$.  Let $T = V(H)$, let $S$ be the set of attachment points of $H$ on $P$, let $k=|S|$, and let $t = |T|$.

Note that $H$ is a complete graph, or else an induced copy of $P_3$ in $H$ together with an endpoint of $P$ would induce a copy of $P_3+P_1$ in $G$. We now claim that $xv_i\in E(G)$ for each $v_i\in S$. Otherwise, by \Cref{lem:x-fiber}, given a neighbor $z$ of $v_i$ in $H$, $\{z, v_i, v_{i+1}, x\}$ would induce a copy of $P_3 + P_1$. 

Next we claim that $v_{i-1}v_{i+1}\not\in E(G)$ when $v_i\in S$.  Otherwise, we obtain a longer path by starting with a neighbor $z$ of $v_i$ in $H$, walking along $z v_i x$, following $P$ from $x$ to $v_{i-1}$, traversing $v_{i-1} v_{i+1}$, and following $P$ from $v_{i+1}$ to $y$. Therefore $zv_i\in E(G)$ for each $z\in T$ and $v_i\in S$, otherwise $\{z,v_{i-1},v_i,v_{i+1}\}$ would induce a copy of $P_3+P_1$.  It follows that $N(z) = (T \setminus \{z\})\cup S$ for each $z\in T$. In particular, $d(u)=(t-1)+k$.

Next we claim that, if $v_i,v_j\in S$ with $i\ne j$, then $v_iv_{j+1}\in E(G)$.  Otherwise, given a neighbor $z$ of $v_i$ in $H$, the set $\{z, v_i, v_{i+1}, v_{j+1}\}$ would induce a copy of $P_3 + P_1$ since $v_{i+1}v_{j+1}\not\in E(G)$ by \Cref{lem:attach-pair}. This implies that, if $v_i\in S$, then the neighborhood of $v_i$ contains $x$, $T$, and $\{v_{j+1}\st v_j\in S\}$, and so $d(v_i)\ge 1 + t + k$.  Therefore $\Delta(G)\ge d(v_i)\ge d(u)+2$, a contradiction. 
\end{proof}

The degree assumption in \Cref{thm:P3+P1} is best possible. Indeed, the complete bipartite graph $K_{t,t+2}$ is $(P_3+P_1)$-free, has maximum degree $t+2$, and the vertices of degree $t$ are not Gallai. 

\subsection{$P_2 + 2P_1$ is a fixer}

\begin{proposition}
If $G$ is a connected $(P_2 + 2P_1)$-free graph, then every vertex of maximum degree is a Gallai vertex.
\end{proposition}
\begin{proof}
Let $G$ be a connected $(P_2 + 2P_1)$-free graph and let $P = v_0\cdots v_\ell$ be a longest path in $G$ with ends $x=v_0$ and $y=v_\ell$.  Suppose for a contradiction that $u$ is a vertex of maximum degree and $u\not\in V(P)$.  Let $k=d(u)=\Delta(G)$, and let $H$ be the component of $G-V(P)$ containing $u$.  Note that $xy\not\in E(G)$, or else we obtain a longer path by starting at a vertex in $H$ with a neighbor on $P$ and traveling around the cycle $P+xy$.  Also, $V(G)-V(P)$ is an independent set, or else, by \Cref{lem:x-fiber}, an adjacent pair of vertices in $V(G)-V(P)$ together with $x$ and $y$ would induce a copy of $P_2+2P_1$.

Let $S$ be the set of attachment points of $H$.  Since $H$ has one vertex, we have $|S|=k$. Applying \Cref{lem:indep-set} where $H$ is the graph with the single vertex $u$, there is an independent set $A\subseteq V(P)$ such that $|A|=k+1$ and $A\cap S=\emptyset$.

If some vertex $s\in S$ has two non-neighbors $w_1,w_2\in A$, then $\{u,s,w_1,w_2\}$ induces a copy of $P_2 + 2P_1$.  Hence every vertex in $S$ has at least $k$ neighbors in $A$.  Counting $u$, every vertex in $S$ has degree at least $k+1$, contradicting that $\Delta(G)=k$.
\end{proof}

Vertices of degree $\Delta(G)-1$ in a $(P_2+2P_1)$-free graph $G$ need not be Gallai. Indeed, consider the graph $G$ obtained from $K_{t,t+2}$ by removing a matching saturating the part of size $t$. $G$ is $(P_2 + 2P_1)$-free and $\Delta(G)=t+1$.  The longest paths in $G$ omit one vertex, and the Gallai vertices are those in the smaller part.  Two of the non-Gallai vertices in the larger part have degree $t$, which equals $\Delta(G)-1$.

\subsection{$4P_1$ is a fixer}\label{sec:indnumthree}

For a path $P$ in a graph $G$ containing the vertices $x$ and $y$, the \emph{closed subpath of $P$ with boundary points $x$ and $y$}, denoted $P[x,y]$, is the subpath of $P$ with endpoints $x$ and $y$.  The \emph{open subpath of $P$ with boundary points $x$ and $y$}, denoted $P(x,y)$, is $P[x,y] - \{x,y\}$.  Additionally, we define the \emph{semi-open} subpaths $P[x,y)$ and $P(x,y]$ analogously. 

Let $x,y\in V(G)$, let $P$ be an $xy$-path in $G$, and let $H$ be a component of $G-V(P)$.  For each non-attachment point $w\in V(P)$, we define the \emph{rank} of $w$, denoted $\rank(w)$, to be the maximum length of a subpath of $P[x,w]$ containing $w$ but no attachment points.  Note that if $s_1,\ldots,s_k$ are the attachment points with indices increasing from $x$ to $y$, then the rank of a non-attachment point $w\in V(P(s_i,s_{i+1}))$ is $\dist_P(s_i,w)-1$.

\begin{lemma}\label{lem:minihammer}
Let $P$ be an $xy$-path in a graph $G$ and let $H$ be a complete component of $G-V(P)$.  Let $S$ be the set of attachment points of $H$ on $P$, where $S=\{s_1,\ldots,s_k\}$, with indices increasing from $x$ to $y$, and suppose that the induced $(S,V(H))$-bigraph has a matching saturating $S_0$ when $S_0\subseteq S$ and $|S_0|\le |V(H)|$.  The following hold.
\begin{enumerate}
    \item If $s_1=x$, then $P$ has an augmenting splice with endpoint $y$.  If $s_k=y$, then $P$ has an augmenting splice with endpoint $x$.  If $s_i$ and $s_{i+1}$ are consecutive on $P$, then $P$ has an augmenting interior splice.  
    
    \item If some component $P_0$ of $P-S$ has fewer than $|V(H)|$ vertices, then $P$ has an augmenting splice replacing $P_0$.  
    
    \item If $w$ and $w'$ are in distinct components of $P - S - V(P[x,s_1])$,  $\rank(w)+\rank(w')<|V(H)|$, and $ww'\in E(G)$, then $P$ has an augmenting detour.
    
    \item If $w$ and $w'$ are in distinct components of $P-S$, $\rank(w)+\rank(w')<|V(H)|$, and $ww'\in E(G)$, then $G$ has a path with endpoint $y$ that is longer than $P$.
\end{enumerate}
\end{lemma}
\begin{proof}
For part 1, if $s_1=x$ or $s_k=y$, then we obtain an augmenting splice of $P$ by prepending or appending a Hamiltonian path of $H$.  If $s_i$ and $s_{i+1}$ are consecutive along $P$, then it follows from \Cref{lem:attach-pair} that $P$ has an augmenting interior splice.  

For part 2, let $P_0$ be a component of $P-S$ with $1\le |V(P_0)|<|V(H)|$.  Note that $P_0$ is $P[x,s_1)$, or $P(s_k,y]$, or $P(s_i,s_{i+1})$ for some $i$.  Suppose that $P_0 = P(s_i,s_{i+1})$.  Hence there is a matching $\{s_iz,s_{i+1}z'\}$ joining $s_i$ and $s_{i+1}$ to distinct vertices $z,z'\in V(H)$.  Since $H$ is complete, $H$ contains a spanning $zz'$-path $Q$. Since $|V(P_0)|<|V(H)|$, we obtain an augmenting interior splice by replacing $P_0$ with $Q$. The cases $P_0=P[x,s_1)$ and $P_0=P(s_k,y]$ are similar, except that we obtain an augmenting external splice.

For part 3, we may assume that $w$ appears before $w'$ when traversing $P$ from $x$ to $y$ (see \Cref{fig:lemma4}).  Let $i$ and $j$ be indices such that $w\in V(P(s_i,s_{i+1}))$ and $w'\in V(P(s_j,s_{j+1}))$ except that we set $j=k$ if $w'\in V(P(s_k,y])$.  Since $w$ and $w'$ are in distinct components of $P-S$, we have $i<j$.  If $|V(H)|=1$, then $\rank(w)+\rank(w')<|V(H)|$ implies that $w$ immediately follows $s_i$ and $w'$ immediately follows $s_j$.  By \Cref{lem:attach-pair} part (2), we have that $P$ has an augmenting detour.  Otherwise, $|V(H)|\ge 2$ and there is a matching $\{s_iz,s_jz'\}$ joining $s_i$ and $s_j$ to distinct vertices $z,z'\in V(H)$.  Let $Q_1$ be an $s_is_j$-path whose interior vertices form a spanning $zz'$-path in $H$, and let $Q_2$ be the path $ww'$.  The detour of $P$ with patching paths $Q_1$ and $Q_2$ adds the vertices in $V(H)$ but omits the $\rank(w)$ vertices in $P(s_i,w)$ and the $\rank(w')$ vertices in $P(s_j,w')$. Since $\rank(w)+\rank(w')<|V(H)|$, the detour is augmenting.

\begin{figure}[h!]
\centering 
\includegraphics[scale=0.9]{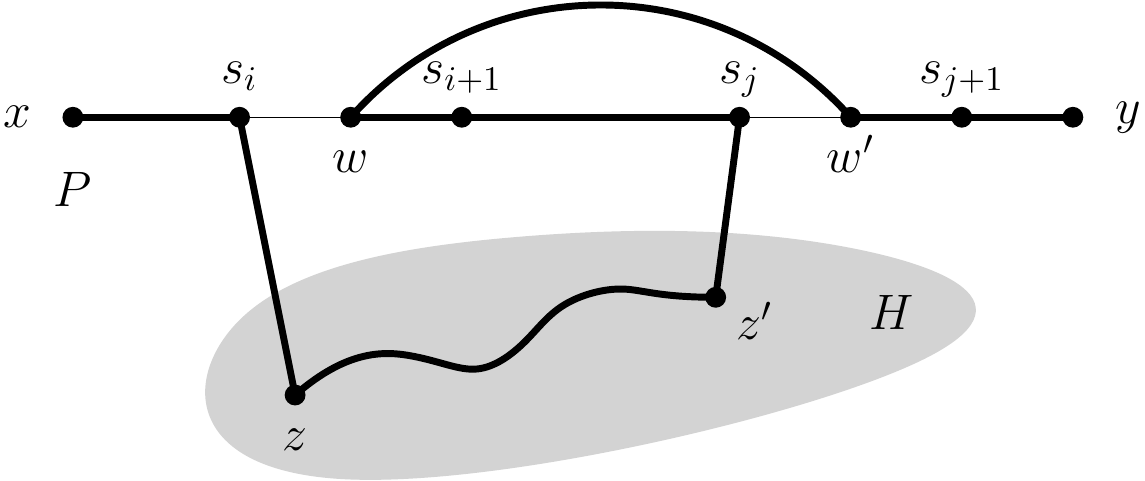}
\caption{Part 3 in \Cref{lem:minihammer}.}
\label{fig:lemma4}
\end{figure}

For part 4, we may apply the argument for part 3 unless $w\in V(P[x,s_1])$. As before, let $j$ be the index such that $w'\in V(P(s_j,s_{j+1}))$, except that we set $j=k$ if $w'\in V(P(s_k,y])$.  We obtain a new path $P'$ by following $P$ backward from $y$ to $w'$, traversing $w'w$, following $P$ forward from $w$ to $s_j$, traversing an edge joining $s_j$ and a vertex in $H$, and finishing with a Hamiltonian path in $H$.  The path $P'$ includes all of $V(H)$ but omits the $\rank(w)$ vertices in $P[x,w)$ and the $\rank(w')$ vertices in $P(s_j,w')$. Since $\rank(w)+\rank(w')<|V(H)|$, the path $P'$ is longer than $P$. 
\end{proof}

Our next lemma provides additional structure when $G$ is $k$-connected and $\alpha(G)\le k+2$.

\begin{lemma}\label{lem:hammer}
Let $P$ be a longest path in a graph $G$ with endpoints $x$ and $y$, and let $H$ be a component of $G-V(P)$.  Suppose that $G$ is $k$-connected and $\alpha(G)\le k+2$.  The following hold.
\begin{enumerate}
    \item The set $S$ of attachment points of $H$ on $P$ has size $k$.
    
    \item The subgraph $H$ is complete.
    
    \item The graph $P-S$ has $k+1$ components, and each has at least $|V(H)|$ vertices.
    
    \item If $w$ and $w'$ are in distinct components of $P-S$ and $\rank(w)+\rank(w')<|V(H)|$, then $ww'\not\in E(G)$.
    
    \item The vertices in each component of $P-S$ of rank less than $|V(H)|$ form a clique.

\end{enumerate}
\end{lemma}
\begin{proof}
Let $S=\{s_1,\ldots,s_r\}$, with indices increasing from $x$ to $y$.  Since $G$ is $k$-connected and $H$ is a component of $G-V(P)$, it follows that $r\ge k$, or else $S$ separates $V(H)$ from $x$ and $y$. Since $P$ is a fiber, it follows from \Cref{lem:indep-set} that $G$ contains an independent set $A$ with $|A|=r+1$ such that $A\subseteq V(P)$ and no edge joins $A$ and $V(H)$.  Since $1 + (k+1)\le \alpha(H)+(r+1)=\alpha(H) + |A| \le \alpha(G) \le k+2$, it follows that $\alpha(H)=1$ and $r=k$.  Hence, there are exactly $k$ attachment points and $H$ is complete.

Let $S_0\subseteq S$ with $|S_0|\le |V(H)|$ and let $B$ be the induced $(S_0,V(H))$-bigraph. If $B$ has no matching saturating $S_0$, then Hall's Theorem \citep{West} implies that there exists $S_1 \subseteq S_0$ such that $|N_B(S_1)|<|S_1|$.  Since $|N_B(S_1)| < |S_1| \le |S_0| \le |V(H)|$, it follows that $N_B(S_1)\cup (S-S_1)$ is a cutset of size less than $k$, contradicting that $G$ is $k$-connected. Therefore \Cref{lem:minihammer} applies, and since $P$ is a longest path, parts 3 and 4 follow.

It remains to establish part 5.  Suppose for a contradiction that $w$ and $w'$ are distinct vertices in the same component $W$ of $P-S$ such that $\rank(w),\rank(w') < |V(H)|$ and $ww'\not\in E(G)$.  Let $A$ be the set of non-attachment points in $P$ with rank 0, and obtain $A'$ from $A$ by deleting the vertex in $W\cap A$ and adding $w$ and $w'$.  Note that, with the possible exception of $\{w,w'\}$, each pair of vertices in $A'$ has rank sum less than $|V(H)|$ and intersects two components of $P-S$.  It follows from part 4 that $A'$ is an independent set.  Since $|A'|=k+2$ and $A$ consists of non-attachment points, we may add any vertex in $H$ to obtain an independent set of size $k+3$, a contradiction.
\end{proof}

\begin{theorem}\label{thm:zero-sided-chvatal}
Let $k\in \{1,2\}$.  If $G$ is $k$-connected and $\alpha(G)\le k+2$, then every longest path in $G$ contains every vertex of degree at least $\Delta(G)-(2-k)$.
\end{theorem}
\begin{proof}
Let $P$ be a longest path in $G$ with endpoints $x$ and $y$, and suppose for a contradiction that there exists $u \notin V(P)$ with $d(u)\ge \Delta(G)-(2-k)$.  Let $H$ be the component of $G-V(P)$ containing $u$, and let $t=|V(H)|$. Let $s_1,\ldots,s_k$ be the attachment points of $H$ on $P$, indexed in order from $x$ to $y$, and let $S=\{s_1,\ldots,s_k\}$.  Note that $\Delta(G) \le d(u) + (2-k) \le ((t-1)+k) + (2-k) = t+1$.

For each component $W$ of $P-S$, let $f(W)$ be the set of vertices $w$ in $W$ with $\rank(w)<t$. We claim that $N(s_1)$ either contains $V(H)$ or $f(W)$, for some component $W$ of $P-S$.  If not, then let $A$ be the set of vertices consisting of the lowest-ranked non-neighbor of $s_1$ in each component of $P-S$.  Note that if $\{w,w'\}$ is a pair of vertices in $A$, then $\rank(w)+\rank(w')<t$, or else $s_1$ has a set $B$ of at least $t$ neighbors in the components of $P-S$ containing $w$ and $w'$.  Let $z$ be the vertex in $P[x,s_1]$ that preceeds $s_1$.  Note that $z\not\in B$, since some non-neighbor of $s_1$ separates $z$ and the initial segment of $P[x,s_1)$ consisting of vertices belonging to $B$.  Counting $B$ together with $z$, it follows that $d(s_1)\ge t+2$, contradicting that $\Delta(G)\le t+1$.  Hence $\rank(w)+\rank(w')<t$ and it follows from \Cref{lem:hammer} part (4) that $A$ is an independent set.  But $A$ together with $s_1$ and a non-neighbor of $s_1$ in $H$ forms an independent set of size $k+3$, contradicting that $\alpha(G)\le k+2$.  Therefore $N(s_1)$ either contains $V(H)$ or $f(W)$ for some component $W$ of $P-S$.

Note that $|V(H)|=t$ and $|f(W)|=t$ for each component $W$ of $P-S$.  Let $v$ and $v'$ be the immediate neighbors of $s_1$ along $P$, and let $v''$ be a neighbor of $s_1$ in $H$.  Noting that $V(H)$ and each $f(W)$ intersect $\{v,v',v''\}$ in at most one vertex, it follows that $d(s_1)\ge t+3 - 1$, contradicting that $\Delta(G)\le t+1$.
\end{proof}

We note two consequences.

\begin{corollary}
If $G$ is a connected graph with $\alpha(G)\le 3$ and $\Delta(G)-\delta(G)\le 1$, or if $G$ is a  $2$-connected regular graph with $\alpha(G)\le 4$, then $G$ has a Hamiltonian path.
\end{corollary}

\begin{corollary}
The graph $4P_1$ is a fixer.
\end{corollary}

\section{A $5$-vertex fixer}\label{sec:5P1}

In this section, we show that $5P_1$ is a fixer.  Although $5P_1$ is a fixer, there are connected $5P_1$-free graphs in which no vertex of maximum degree is Gallai (see \Cref{Ex:BestPossible}).  By contrast, for each fixer $F$ of order at most $4$, the vertices of maximum degree in a connected $F$-free graph are all Gallai: \citet{GS18} show this for $F=2P_2$, our results in \Cref{monogenic} show this for $F\in\{P_3+P_1, P_2+2P_1, 4P_1\}$, and we leave the case $F=P_4$ as an exercise.  

The statement that $5P_1$ is a fixer is equivalent to the statement that if $G$ is a connected graph with $\alpha(G)\le 4$, then $G$ has a Gallai vertex.   In the case that $G$ is $2$-connected, the result already follows from \Cref{thm:zero-sided-chvatal}.  When $G$ has cut-vertices, we exploit the block-cutpoint structure of $G$.  We need the following two variants of \Cref{thm:zero-sided-chvatal} in the case that $P$ is an $x$-fiber or an $xy$-fiber for distinguished vertices $x,y\in V(G)$.

\begin{lemma}\label{lem:one-sided-chvatal}
Let $G$ be a $2$-connected graph with a distinguished vertex $x$.  If $\alpha(G-x)\le 3$, then every $x$-fiber contains every vertex in $G$ of maximum degree.
\end{lemma}
\begin{proof}
Let $P$ be an $x$-fiber with other endpoint $y$, and suppose for a contradiction that $u$ is a vertex of maximum degree not on $P$.  Let $H$ be the component of $G-V(P)$ containing $u$, and let $r$ be the number of attachment points of $H$ on $P$. Note that $r\ge 2$, or else there is at most one attachment point separating $y$ and $H$, contradicting that $G$ is $2$-connected.  Moreover, by \Cref{lem:indep-set} part (2), we have that $r+\alpha(H)\le \alpha(G-x) \le 3$. Since $r\ge 2$ and $\alpha(H)\ge 1$, it follows that $r=2$ and $\alpha(H)=1$.  Therefore $H$ is a complete graph. Let $\{s_1,s_2\}$ be the set of attachment points of $H$ on $P$, with indices increasing from $x$ to $y$, and let $S=\{s_1,s_2\}$.

Since $G$ is $2$-connected, there is a matching in the induced $(S,V(H))$-bigraph saturating $S$ or $|V(H)|=1$.  Let $t=|V(H)|$ and note that $d(u)\le (t-1)+2=t+1$.  Since $P$ is an $x$-fiber, it follows from \Cref{lem:minihammer} that both $P(s_1,s_2)$ and $P(s_2,y]$ are non-empty (part (1)) and have at least $t$ vertices (part (2)).  If $s_2$ has at least $t$ neighbors in some set in $\{V(H),V(P(s_1,s_2)),V(P(s_2,y])\}$, then $d(s_2)\ge t + 2 > d(u)$, contradicting that $u$ has maximum degree.  Hence $s_2$ has fewer than $t$ neighbors in each of $V(H)$, $V(P(s_1,s_2))$, and $V(P(s_2,y])$.  Let $w_1$ and $w_2$ be the non-neighbors of $s_2$ of minimum rank in $P(s_1,s_2)$ and $P(s_2,y]$, respectively, and let $z$ be a non-neighbor of $s_2$ in $H$.  

We claim that $\{s_2,z,w_1,w_2\}$ is an independent set, contradicting $\alpha(G-x)\le 3$.  By construction, $s_2$ has no neighbor in $\{z,w_1,w_2\}$.  Since $w_1$ and $w_2$ are not attachment points, $z$ has no neighbor in $\{w_1,w_2\}$.  If $w_1w_2\in E(G)$, then \Cref{lem:minihammer} part (3) and the fact that $P$ is an $x$-fiber imply that $\rank(w_1)+\rank(w_2)\ge t$.  Hence $s_2$ is adjacent to all vertices in $P(s_1,w_1)$ and $P(s_2,w_2)$, and there are at least $t$ of them. Together with the vertex preceding $s_2$ in $P$ and a neighbor of $s_2$ in $H$, we have $d(s_2)\ge t+2$, contradicting that $u$ has maximum degree.
\end{proof}

\begin{lemma}\label{lem:two-sided-chvatal}
Let $G$ be a $2$-connected graph and let $x$ and $y$ be distinct vertices of $G$.  If $\alpha(G-\{x,y\})\le 2$, then every $xy$-fiber contains every vertex in $G$ of maximum degree or $G-\{x,y\}$ is the disjoint union of two complete graphs.
\end{lemma}
\begin{proof}
Let $P$ be an $xy$-fiber, let $u$ be a vertex of maximum degree not on $P$, and let $H$ be the component of $G-V(P)$ containing $u$.  Let $\{s_1,\ldots,s_r\}$ be the set of attachment points of $H$, with indices increasing from $x$ to $y$, and let $S=\{s_1,\ldots,s_r\}$.  Since $G$ is $2$-connected, we have $r\ge 2$, or else deleting $S$ separates $H$ from $V(P)-S$ (which is non-empty since $x\ne y$).  By \Cref{lem:indep-set}, there is an independent set $A\subseteq V(P-\{x,y\})$ such that $|A|=r-1$ and there are no edges joining $A$ and $V(H)$.  Therefore $1 + 1 \le (r-1) + \alpha(H) \le \alpha(G-\{x,y\}) \le 2$.  It follows that $r=2$ and $\alpha(H)=1$.

Let $t=|V(H)|$.  Note that $H$ is complete and, since $G$ is $2$-connected, there is a matching in the induced $(S,V(H))$-bigraph saturating $S$ or $|V(H)|=1$.  By \Cref{lem:minihammer}, we have $|V(P(s_1,s_2))|\ge t$ or else there is an augmenting interior splice of $P$ replacing $P(s_1,s_2)$, contradicting that $P$ is an $xy$-fiber.

Let $W=V(P(s_1,s_2))$.  Note that $W$ is a clique, or else a non-adjacent pair of vertices in $W$ together with a vertex in $H$ gives an independent set of size $3$, contradicting $\alpha(G-\{x,y\})\le 2$.

If $(x,y)=(s_1,s_2)$, then $G-\{x,y\}$ is the disjoint union of the complete graph $H$ and the complete graph on $W$.  Otherwise, if $x\ne s_1$, then $s_1$ has a non-neighbor in $H$ and a non-neighbor in $W$, or else $d(s_1)\ge t+2>d(u)$.  So $s_1$ together with a non-neighbor in $W$ and a non-neighbor in $H$ form an independent set of size $3$ in $G-\{x,y\}$, a contradiction.  The case that $y\ne s_2$ is similar.
\end{proof}

The \emph{block-cutpoint graph} of a graph $G$ is a bipartite graph $H$ in which one part consists of the cut-vertices of $G$ and the other has a vertex $b_i$ for each block $B_i$ of $G$. Moreover, $vb_i$ is an edge of $H$ if and only if $v \in B_i$. When $G$ is connected, its block-cutpoint graph is a tree whose leaves are the blocks of $G$ (see, e.g., \cite{West}). We say that a block $B$ of a graph $G$ is \emph{special} if every longest path in $G$ contains an edge in $B$. 

\begin{lemma}\label{lem:special-block}
If no cut-vertex in a connected graph $G$ is Gallai, then $G$ has a special block.
\end{lemma}
\begin{proof}
Let $G$ be a connected graph such that no cut-vertex is Gallai. Suppose for a contradiction that no block of $G$ is special.  Let $T$ be the block-cutpoint tree of $G$. We construct a digraph $D$ on $V(T)$ in which each vertex has out-degree $1$.  Let $B$ be a block in $G$. We identify a particular cut-vertex $x\in V(B)$ and we include the directed edge $Bx$ in $D$.  Since $B$ is not special, some longest path of $G$ is contained in some component $H$ of $G-E(B)$.  Note that $H$ and $B$ have exactly one vertex in common, and we take $x$ to be this cut-vertex.  

Let $x$ be a cut-vertex in $G$.  We specify a particular block $B$ that contains $x$ and we include the directed edge $xB$ in $D$.  Since $x$ is not Gallai, some component $H$ of $G-x$ contains a longest path in $G$.  Let $B$ be the block containing $x$ such that $B-x\subseteq H$.  We add the directed edge $xB$ to $E(D)$.  

Since $|E(D)| = |V(T)| > |E(T)|$, it follows that there is a block $B$ and a cut-vertex $x$ such that both $Bx$ and $xB$ are edges in $D$.  This implies that $G$ has vertex-disjoint longest paths, contradicting the fact that every two longest paths in a connected graph share at least one vertex.
\end{proof}

\begin{lemma}\label{lem:special-to-Gallai}
If $G$ is a connected graph, $\alpha(G)\le 4$, and $G$ has a special block, then $G$ has a Gallai vertex.
\end{lemma}
\begin{proof}
Let $G$ be a connected graph with $\alpha(G) \le 4$ and with a special block $B$. Let $S$ be the set of cut-vertices in $B$, with $S=\{x_1,\ldots,x_k\}$.  Since $\alpha(G)\le 4$, we have $k\le 4$. 

Case $k=0$.  In this case, $G=B$ and so $G$ is $2$-connected.  It follows from \Cref{thm:zero-sided-chvatal} that $G$ has a Gallai vertex.

Case $k=1$.  Let $u\in V(B)$ with $d_B(u)= \Delta(B)$.  We claim that $u$ is a Gallai vertex in $G$.  Let $P$ be a longest path in $G$.  If $P$ is contained in $B$, then $u\in V(P)$ by \Cref{thm:zero-sided-chvatal}.  If $P$ leaves $B$ through the cut-vertex $x_1$, then $P\cap B$ is an $x_1$-fiber in $B$ and it follows that $u\in V(P)$ by \Cref{lem:one-sided-chvatal}.

Case $k=2$. Suppose first that $B-S$ is not the disjoint union of two complete graphs.  Let $u\in V(B)$ with $d_B(u)=\Delta(B)$.  We claim that $u$ is a Gallai vertex.  Let $P$ be a longest path in $G$.  Since $B$ is special, it follows that $P\cap B$ is a nontrivial subpath of $P$.  Note that, as a subgraph of $B$, the path $P\cap B$ is either a fiber, an $x_1$-fiber or an $x_2$-fiber, or an $x_1x_2$-fiber, depending on whether $P$ has two, one, or zero endpoints in $B$, respectively.  It follows from \Cref{thm:zero-sided-chvatal}, \Cref{lem:one-sided-chvatal}, or \Cref{lem:two-sided-chvatal} that $u\in V(P\cap B)$, respectively.

Otherwise, suppose that $B-S$ is the disjoint union of two complete graphs $W_1$ and $W_2$ (see \Cref{fig:lemma19}).  Since $B$ is $2$-connected, for $i\in\set{1,2}$, there is a matching in the induced $(S,V(W_i))$-bigraph saturating $S$ or $|V(W_i)|=1$.  Also, since $S$ is a minimum cut in $B$, each vertex in $S$ has neighbors in $V(W_1)$ and $V(W_2)$.  It follows that $B$ has a Hamiltonian cycle.  We claim that $x_2$ is a Gallai vertex.  Let $P$ be a longest path in $G$, and suppose for a contradiction that $x_2\not\in V(P)$. Since $B$ is special, $P$ has at least one endpoint in $B$.  Replacing the subpath of $P$ inside $B$ with an appropriate Hamiltonian path gives a longer path in $G$.

\begin{figure}[h!]
\centering 
\includegraphics[scale=0.9]{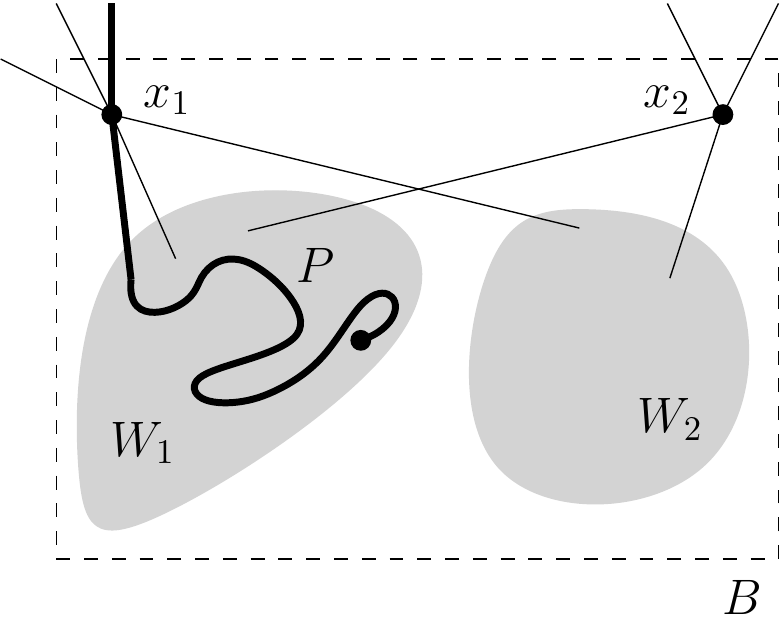}
\caption{Case $k = 2$ in the proof of \Cref{lem:special-to-Gallai}.}
\label{fig:lemma19}
\end{figure}

Case $k=3$. Note that $B-S$ is a complete graph $W_1$ or else $\alpha(G)> 4$.  Suppose there is a pair of cut-vertices, say $\{x_1,x_3\}$, such that $B-\{x_1,x_3\}$ is the disjoint union of two complete graphs.  These are necessarily $W_1$ and the $1$-vertex subgraph consisting of $x_2$; let $W_2$ be this $1$-vertex subgraph.  As in the case $k=2$, it follows that $B$ has a Hamiltonian cycle containing $x_1x_2x_3$ as a subpath.  We claim that $x_3$ is a Gallai vertex.  Let $P$ be a longest path in $G$ and suppose for a contradiction that $x_3\not\in V(P)$.  Note that $P$ cannot have an endpoint in $B$, or else replacing $P\cap B$ with an appropriate Hamiltonian path gives a longer path in $G$. Therefore, as a subgraph of $B$, the path $P\cap B$ is an $x_1x_2$-fiber.  But $B$ has a spanning $x_1x_2$-path, contradicting $x_3\not\in V(P)$.  

Otherwise, there is no pair of cut-vertices whose removal from $B$ results in the disjoint union of two complete graphs.  Let $u\in V(B)$ with $d_B(u)=\Delta(B)$.  We claim that $u$ is a Gallai vertex.  Let $P$ be a longest path in $G$.  It follows that, as a subgraph of $B$, the path $P\cap B$ is a fiber, an $x_i$-fiber for some $x_i\in S$, or an $x_ix_j$-fiber for some $x_i,x_j\in S$, depending on whether $P$ has two, one, or zero endpoints in $B$, respectively.  It follows from \Cref{thm:zero-sided-chvatal}, \Cref{lem:one-sided-chvatal}, or \Cref{lem:two-sided-chvatal} that $u\in V(P\cap B)$, respectively.

Case $k=4$.  The condition $\alpha(G)\le 4$ requires that $|V(B)|=4$ and $2$-connectivity requires that $B$ contains a $4$-cycle $C$.  Let $x_i$ be a cut-vertex in $B$ which maximizes the length of an $x_i$-fiber in $G-E(B)$.  We claim that $x_i$ is a Gallai vertex.  Let $P$ be a longest path in $G$, and suppose for a contradiction that $x_i\not\in V(P)$.  The path $P$ decomposes into three subpaths $P_1$, $P_2$, and $P_3$, where $P_2=P\cap B$.  Let $x_j$ be the vertex in $V(P_1)\cap V(P_2)$, and let $x_k$ be the vertex in $V(P_2)\cap V(P_3)$.  Since $|V(B)| = 4$, it follows that $x_j$ or $x_k$ is a neighbor of $x_i$ in $C$.  If $x_kx_i\in E(C)$, then we find a longer path in $G$ by keeping $P_1$, extending $P_2$ by the edge $x_kx_i$ to obtain $P'_2$, and replacing $P_3$ with an $x_i$-fiber $P'_3$ in $G-E(B)$.  Since $P'_2$ is longer than $P_2$ and $P'_3$ is at least as long as $P_3$ by our choice of $x_i$, the path obtained by combining $P_1$, $P'_2$, and $P'_3$ is longer than $P$.  The case $x_jx_i\in E(C)$ is symmetric.
\end{proof}

Applying our lemmas gives the following.

\begin{theorem}\label{thm:5P1}
Let $G$ be a connected graph. If $\alpha(G)\le 4$, then $G$ has a Gallai vertex.  Equivalently, $5P_1$ is a fixer.
\end{theorem}
\begin{proof}
If some cut-vertex in $G$ is Gallai, then the claim follows.  Otherwise, we have that $G$ has a special block by \Cref{lem:special-block}, and hence $G$ has a Gallai vertex by \Cref{lem:special-to-Gallai}.
\end{proof}

The graph $G_0$ from \Cref{fig:counter} shows that there is a connected graph $G$ such that $G$ has no Gallai vertex and $\alpha(G)=6$.  The case $\alpha(G)\le 5$ remains open.

\begin{conjecture}\label{conj:indep-5}
If $\alpha(G)\le 5$ and $G$ is connected, then $G$ has a Gallai vertex.
\end{conjecture}

When $G$ is $3$-connected, $\alpha(G)\le 5$, and $G$ is sufficiently large, \Cref{thm:ChvatalErdos} shows that $G$ has a Gallai vertex.  Outside of a finite number of cases when $\kappa(G)\ge 3$, resolving \Cref{conj:indep-5} reduces to the cases that $\kappa(G)=1$ and $\kappa(G)=2$.  Although it is reasonable to expect that the case $\kappa(G)=1$ may be treated by analyzing the block structure of $G$, it is less clear how to handle the case $\kappa(G)=2$.

\section{A Chv\'atal--Erd\H{o}s type result}\label{sec:CELR}

A celebrated result of \citet{CE72} states that if $\alpha(G)\le \kappa(G)$, then $G$ has a Hamiltonian cycle, and the same technique shows that $G$ has a Hamiltonian path when $\alpha(G) \le \kappa(G)+1$.  Clearly, when $G$ has a Hamiltonian path, every vertex in $G$ is Gallai.  We show that if $\alpha(G)\le \kappa(G)+2$ and $G$ is sufficiently large in terms of $\kappa(G)$, then the maximum degree vertices in $G$ are Gallai. 
  
\begin{theorem}\label{thm:ChvatalErdos}
For each positive integer $k$, there exists an integer $n_0$ such that if $G$ is an $n$-vertex $k$-connected graph with $\alpha(G)\le k+2$ and $n\ge n_0$, then each vertex of maximum degree is Gallai.
\end{theorem}

\begin{proof}
We take $n_0 = k(k+2)(2k+3) + 1$.  Let $P$ be a longest path in $G$ with endpoints $x$ and $y$, and suppose for a contradiction that $u\in V(G)-V(P)$ and $d(u)=\Delta(G)$.  Let $H$ be the component of $G-V(P)$ containing $u$, and let $t = \abs{V(H)}$.  From \Cref{lem:hammer}, it follows that $H$ is complete and $H$ has a set $S$ of $k$ attachment points on $P$.  Let $S=\{s_1,\ldots,s_k\}$ with indices increasing from $x$ to $y$.  For $1\le i< k$, let $W_i = V(P(s_i,s_{i+1}))$; we also define $W_0 = V(P[x,s_1))$ and $W_k=V(P(s_k,y])$.  By \Cref{lem:hammer}, we have that $\abs{W_i}\ge t$ for $0\le i\le k$.  Since $u\in V(H)$, we have that $N(u)\subseteq (V(H)-\{u\}) \cup S$ and therefore $\Delta(G)=d(u)\le (t-1) + k$.  If $t\le 2k(k+1)$, then $\Delta(G)\le k(2k+3)-1$ and so $\alpha(G)\ge n/(\Delta(G) + 1) \ge n/[k(2k+3)] > k+2$, since $n\ge n_0$.  Therefore we may assume that $t > 2k(k+1)$.

We claim that $H$ is the only component of $G-V(P)$.  If $G-V(P)$ contains a second component $H'$, then let $S'$ be the set of attachment points of $H'$ on $P$.  By \Cref{lem:hammer}, it follows that $|S'|=k$.  For each $i$, choose $a_i\in W_i$ among the vertices with ranks in $\{0,\ldots,k\}$ so that $a_i\not\in S'$.  Let $A=\{a_0,\ldots,a_k\}$.  Since $t > 2k(k+1) > 2k$, it follows from \Cref{lem:hammer} that $A$ is an independent set of size $k+1$.  Since $A$ is disjoint from $S\cup S'$, we may extend $A$ to an independent set of size $k+3$ by adding a vertex in $H$ and a vertex in $H'$.  Since $\alpha(G)\le k+2$, we obtain a contradiction, and so $H$ is the only component of $G-V(P)$.

Next, we claim that each vertex $w\in W_i$ has at most $k$ neighbors outside $W_i$.  Let $A$ be the subset of $V(P)-S$ consisting of the vertices $w$ such that $\rank(w)=0$.  By \Cref{lem:hammer}, we have that $A$ is an independent set with $|A|=k+1$.  Note that each vertex $w\in V(P) - (S\cup A)$ has at least one neighbor in $A$, or else $w$ together with $A$ and a vertex in $H$ would give an independent set of size $k+3$.  Since $\abs{A}=k+1$ and $\Delta(G)\le t+k-1$, it follows that $|V(P) - (S\cup A)| \le (k+1)(t+k-1)$ and hence $|V(P)-S|\le (k+1)(t+k)=t(k+1)+k(k+1)$.  Since $V(P)-S = \bigcup_{i=0}^k W_i$ and $|W_i|\ge t$ for each $i$, it follows that $t\le |W_i| \le t+k(k+1)$.  By \Cref{lem:hammer}, in each $W_i$, the $t$ vertices of smallest rank form a clique.  By symmetry, in each $W_i$, the $t$ vertices of largest rank also form a clique.  Since $|W_i|\le t+k(k+1)<2t$, it follows that each vertex in $W_i$ is among the $t$ vertices with smallest rank or the $t$ vertices with largest rank.  In particular, each vertex in $W_i$ has at least $t-1$ neighbors in $W_i$ and hence at most $k$ neighbors outside $W_i$.

It now follows that each $W_i$ is a clique.  Indeed, if $w_i,w'_i\in W_i$ but $w_iw'_i\not\in E(G)$, then we obtain an independent set $A$ with $A\subseteq V(P)-S$ and $|A| = k+2$ as follows.  Starting with $A = \{w_i,w'_i\}$, we add a vertex to $A$ from each $W_j$ with $j\ne i$.  Since $|W_j| \ge t > k(k+1)$ and each of the vertices already in $A$ have at most $k$ neighbors in $W_j$, some vertex in $W_j$ can be added to $A$.  The set $A$ together with a vertex in $H$ gives an independent set of size $k+3$, a contradiction.  Hence each $W_i$ is a clique.  

A vertex $z$ \emph{dominates} a set of vertices $B$ if $z$ is adjacent to each vertex in $B$.  Next, we claim that each $s_i\in S$ dominates some set in  $\{W_0,\ldots,W_k,V(H)\}$.  If some attachment point $s_i$ has more than $k^2$ non-neighbors in each $W_j$ and a non-neighbor $v$ in $H$, then we may obtain an independent set of size $k+3$ by starting with $\{s_i,v\}$ and adding one vertex from each $W_j$.  It follows that each $s_i$ has at least $t-k^2$ neighbors in some set in $\{W_0,\ldots,W_k,V(H)\}$.  Let $W_{k+1} = V(H)$, let $s_i$ be an attachment vertex, and choose $j$ such that $0\le j\le k+1$ and $s_i$ has at least $t-k^2$ neighbors in $W_j$.  We claim that $s_i$ dominates $W_j$.  Indeed, if $w\in W_j$ but $s_iw\not\in E(G)$, then we obtain an independent set $A$ of size $k+3$ starting with $A = \{s_i,w\}$ and adding one vertex from each $W_\ell$ with $0\le \ell\le k+1$ and $\ell\ne j$.  Since $s_i$ has at most $(t+k-1) - (t-k^2)$ neighbors in $W_\ell$, each of the other vertices already in $A$ has at most $k$ neighbors in $W_\ell$, and $|W_\ell| \ge t > (k(k+1) - 1) + (k+1)k$, it follows that $W_\ell$ contains a vertex that can be added to $A$.  Since $\alpha(G)\le k+2$, we obtain a contradiction, and so $s_i$ dominates $W_j$.

Let $1\le i < k$.  Since $W_i$ is a clique and $W_i = V(P(s_i,s_{i+1}))$, we obtain a path $P'$ with $V(P)=V(P')$ and the same set of attachment points by reordering the vertices in $W_i$ arbitrarily, so long as the first vertex is adjacent to $s_i$ and the last vertex is adjacent to $s_{i+1}$.  Similarly, we may reorder $W_0$ provided that the last vertex in $W_0$ is adjacent to $s_1$ and we may reorder $W_k$ provided that the first vertex in $W_k$ is adjacent to $s_k$.  Let $R$ be the set of neighbors of $S$ in $P$.  Note that for each $w\in W_i - R$ and each $q$ with $1\le q\le |W_i| - 2$, we may obtain a path $P'$ with $V(P) = V(P')$ and the same attachment points in which $\rank(w)=q$ by an appropriate reordering of $W_i$.  It follows that if $ww'\in E(G)$, for some $w\in W_i$ and $w'\in W_j$, with $i$ and $j$ distinct in $\{0,\ldots,k\}$, then $w,w'\in R$. Otherwise, we may reorder $W_i$ and $W_j$ to obtain a new path $P'$ in which either $\rank(w) \le 1$ and $\rank(w')\le 1$, or $\rank(w)\ge |W_i| - 2$ and $\rank(w')\ge |W_j| - 2$.  In the latter case, reversing $P'$ gives a path $P''$ in which $\rank(w) \le 1$ and $\rank(w')\le 1$.  This contradicts \Cref{lem:hammer} with respect to $P'$ or $P''$ since $\rank(w)+\rank(w')\le 2$ but $|V(H)| = t > 2k(k+1)\ge 4$.  

We obtain a final contradiction by showing that some attachment point has degree exceeding $\Delta(G)$.  Let $D=\sum_{i=1}^k d(s_i)$ and note that $D\le k(t+k-1)$.  We give a lower bound on $D$ using three sets of edges.  First, for each $s_i$, let $T_i$ be a set of $3$ edges incident to $s_i$ consisting of the edges joining $s_i$ to its two neighbors in $R$ and a third edge joining $s_i$ and a vertex in $H$.  Second, for $0\le i\le k$, there is a matching $M_i$ of size $k$ joining vertices in $W_i$ and $V(G)-W_i$, or else the K\H{o}nig-Egerv\'{a}ry Theorem \citep{West} implies that the induced $(W_i,V(G)-W_i)$-bigraph has a vertex cover of size less than $k$, which is also a vertex cut since $|W_i|,|V(G)-W_i|\ge t > k$. Obtain $M'_i$ from $M_i$ by discarding edges incident to vertices in $W_i\cap R$. Note that $|M'_i| \ge |M_i| - 2 \ge k-2$ always, but for $i\in\{0,k\}$ we have $|M'_i| \ge |M_i| - 1 \ge k-1$.  Suppose that $e\in M'_i$, let $w$ be the endpoint of $e$ in $W_i$, and let $v$ be the other endpoint of $e$ in $V(G)-W_i$.  Since $w$ is not an attachment point, we have $v\not\in V(H)$, and since $H$ is the only component of $G-V(P)$, it follows that $v\in V(P)-W_i$.  Since $w\not\in R$, it follows that $v$ must be an attachment point.  Hence each edge in $M'_i$ joins a vertex in $W_i - R$ and a vertex in $S$.  Moreover, $M'_i$ and $T_j$ are disjoint, as each edge in $T_j$ has an endpoint in $R\cup V(H)$ and no edge in $M'_i$ has such an endpoint.  With $Z=\bigcup_{i=0}^k M'_i \cup \bigcup_{j=1}^k T_j$, we have $|Z| \ge [(k-1)(k-2) + 2(k-1)] + 3k = k(k+2)$.  Third, for $1\le i\le k$, let $F_i$ be the set of edges joining $s_i$ and a set in $\{W_0,\ldots,W_k,V(H)\}$ dominated by $s_i$.  Note that $|F_i \cap Z| \le 2$, since $F_i$ contains at most one edge in $\bigcup_{i=0}^k M'_i$ and at most one edge in $\bigcup_{j=1}^k T_j$.  Let $F=\bigcup_{j=1}^k F_i$, and note that $|F|\ge tk$ and $|F\cap Z| \le 2k$.

We compute $D \ge |F \cup Z| = |F| + |Z| - |F\cap Z| \ge tk + k(k+2) - 2k = tk + k^2 = k(t+k)$.  Since $D\le k(t+k-1)$, it follows that $k(t+k) \le D\le k(t+k-1)$, contradicting that $k$ is positive.
\end{proof}

\begin{example}\label{Ex:BestPossible}
The assumption $\alpha(G)\le\kappa(G)+2$ in \Cref{thm:ChvatalErdos} is best possible. Let $G$ be the graph obtained from the star $K_{1,k+2}$ with leaves $\{x_1,\ldots,x_{k+2}\}$ by replacing the center vertex with a $k$-clique $S$ and replacing each leaf vertex $x_i$ with a $t$-clique $X_i$ containing a set of $k$ distinguished vertices $Y_i$ that are joined to $S$.  Since $V(G)$ can be covered by $k+3$ cliques, we have $\alpha(G)\le k+3$.  Also, we have $\kappa(G)=k$ since $S$ is a cutset of size $k$ and when $R\subseteq V(G)$ and $|R|<k$, the graph $G-R$ contains at least one vertex in each of $S,Y_1,\ldots,Y_{k+2}$, implying that $G-R$ is connected.

We claim that the set of Gallai vertices in $G$ is $S$.  Since $|S|=k$ and $G-S$ is the disjoint union of $k+2$ copies of $K_t$, it follows that every path in $G$ has at most $|V(G)| - t$ vertices.  Paths in $G$ that achieve this bound contain $S$ and all but one of $X_1,\ldots,X_{k+2}$, implying that $u\in V(G)$ is Gallai if and only if $u\in S$.  By construction, each vertex in $S$ has degree $k(k+2)+(k-1)$.  Hence, when $t$ is sufficiently large, the set of vertices in $G$ of maximum degree is $Y_1\cup \cdots \cup Y_{k+2}$, and none of these is Gallai.

Although maximum degree vertices are not Gallai, our construction still has Gallai vertices. It is natural to ask whether every graph with sufficiently high connectivity has a Gallai vertex \citep{Zam72,Zam01}.  As noted in \Cref{intro}, there are $k$-connected graphs having no Gallai vertices when $k\le 3$. The question remains open for $k\ge 4$.
\end{example}

The complete bipartite graphs $K_{s,s+2}$ show that the condition $\alpha(G) \le \kappa(G)+1$ cannot in general be relaxed to $\alpha(G) \le \kappa(G)+2$ while still guaranteeing existence of Hamiltonian paths \citep{CE72}. However, \Cref{thm:ChvatalErdos} immediately implies that this is possible for sufficiently large regular graphs.

\begin{corollary}\label{hamreg} For each positive integer $k$, there exists $n_0$ such that every $k$-connected regular graph $G$ with $\alpha(G) \leq k+2$ and $n\geq n_0$ vertices has a Hamiltonian path.
\end{corollary}

We do not know whether the condition $\alpha(G) \leq k+2$ in \Cref{hamreg} is best possible. The following construction from \citep{CO13} shows that it cannot be relaxed to $\alpha(G) \leq k+5$.   

\begin{example}\label{Ex:HamReg} Let $k \geq 6$ be even. Let $G_1$ be $K_{k+1}$ minus an edge and let $G_2$ be $K_{k+1}$ minus a matching on $k-4$ vertices. Let $G$ be the graph obtained from two copies of $G_1$ and one copy of $G_2$ by adding a new vertex adjacent to all $k$ vertices of degree $k-1$. We have that $G$ is a $1$-connected regular graph with $\alpha(G) = 6$ and no Hamiltonian path.  
\end{example}

\section{Concluding remarks and open problems}
\newcommand{\HH}{\mathcal{H}}
In this paper we aimed at characterizing monogenic Gallai families. Let $\HH$ be the set of fixers, and recall that $H\in \HH$ if and only if $\Free(H)$ is a Gallai family. We showed that $\HH$ contains $5P_1$ (\Cref{thm:5P1}) and all linear forests on at most $4$ vertices (\Cref{monogenic}).  Also, $\HH$ is contained in the family of linear forests that are induced subgraphs of $G_0$ (\Cref{prop:linforest}).  It remains open to decide if $H\in \HH$ in finitely many cases:

\begin{question} Let $H$ be a linear forest induced subgraph of $G_0$ such that $5 \le |V(H)| \le 9$ and $H \ne 5P_1$.  Is $\Free(H)$ a Gallai family? 
\end{question}

We believe that $\Free(6P_1)$ provides an affirmative answer (\Cref{conj:indep-5}).  It turns out that $3P_3$ and $P_7 + 2P_1$ are the only linear forest induced subgraphs of $G_0$ on $9$ vertices, and hence the only candidates for $9$-vertex fixers, as shown in the following.

\begin{figure}[h!]
\centering 
\includegraphics[scale=0.6]{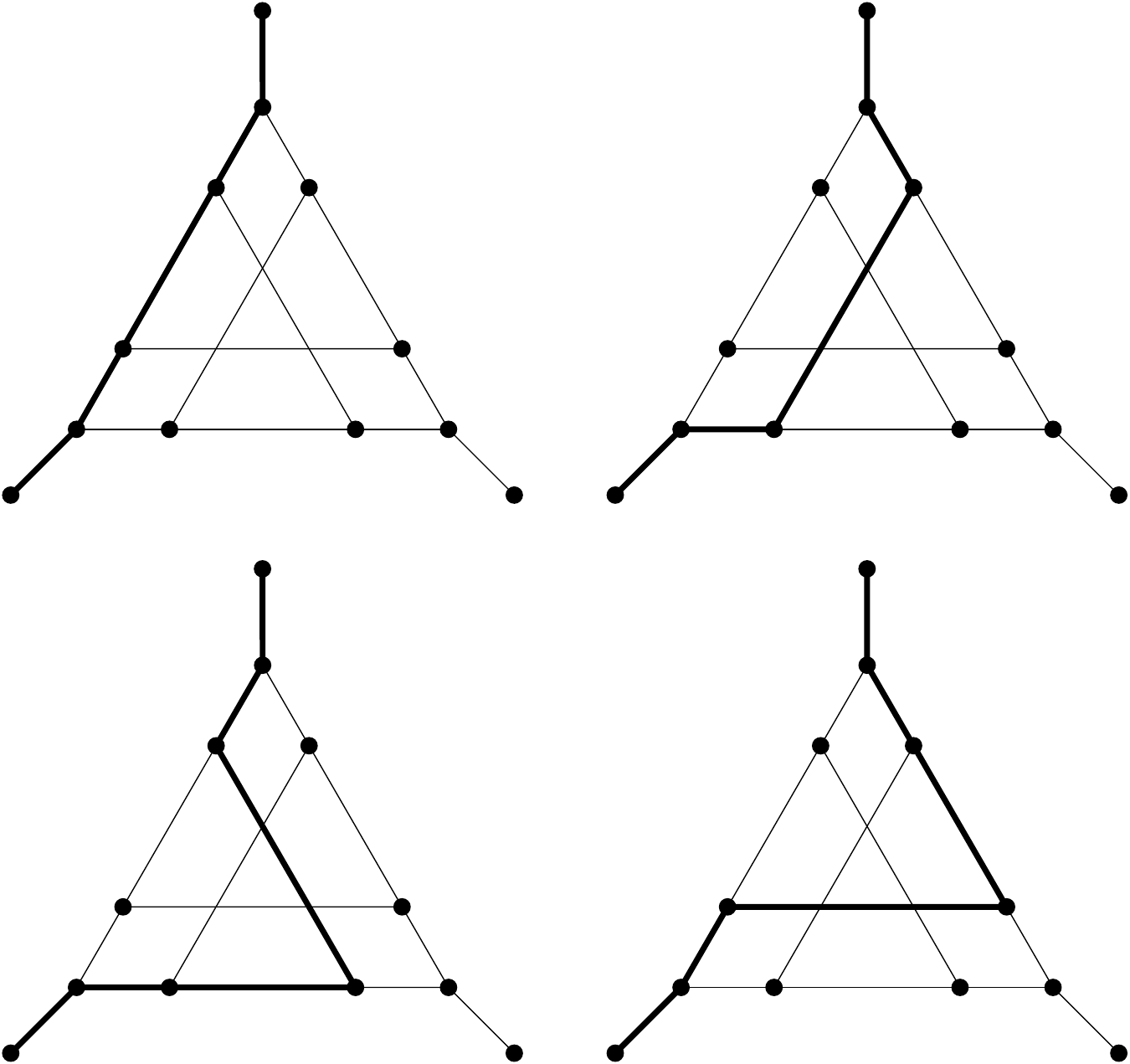}
\caption{The induced paths in $G_0$ containing two degree-$1$ vertices of $G_0$.}
\label{fig:9vert}
\end{figure}

\begin{remark} The graphs $3P_3$ and $P_7 + 2P_1$ are the only $9$-vertex linear forest induced subgraphs of $G_0$. The argument is as follows. Let $H$ be an induced linear forest of $G_0$ on $9$ vertices and let $P=v_{1}\cdots v_{i}$ be a longest path in $H$.

Suppose first that $P$ contains two vertices of degree $1$ in $G_0$.  Since the first $3$ vertices and the last $3$ vertices of $P$ determine $P$, it is easy to see that, up to symmetry, $P$ is one of the bold paths depicted in \Cref{fig:9vert}.  It follows that $H$ is a copy of $P_7 + 2P_1$.

Suppose finally that $P$ contains at most one vertex of degree $1$ in $G_0$. We claim that $i \leq 3$. Indeed, if $i \geq 4$, then $P$ contains at least $i-1 \geq 3$ vertices of degree $3$ in $G_0$, say without loss of generality $v_{1}, v_{2}, v_{3}$.  Note that $v_{1}$ has two neighbors in $V(G_0) - V(H)$ and both $v_{2}$ and $v_{3}$ have one neighbor in $V(G_0) - V(H)$. Since $G_0$ has girth $5$, these neighbors are distinct and so $|V(H)| \leq 12 - 4 = 8$, a contradiction. Suppose now $H$ has $k$ components.  Note that $H$ has $9-k$ edges and $G_0-E(H)$ has $6+k$ edges, each of which has an endpoint in $V(G_0)-V(H)$.  Since $G_0$ is subcubic and $|V(G_0) - V(H)| = 12 - 9 = 3$, it follows that $6+k\le 3\cdot 3$, and so $k \leq 3$.  Hence $H=3P_3$. 
\end{remark}
In \Cref{hamreg}, we observed the following Chv\'atal--Erd\H{o}s type result: for a regular graph $G$, if $\alpha(G) \leq \kappa(G)+2$ and $G$ is sufficiently large in terms of $\kappa(G)$, then $G$ contains a Hamiltonian path. We also observed that we cannot relax $\alpha(G) \leq \kappa(G)+2$ to $\alpha(G) \leq \kappa(G)+5$ and we conclude by asking to determine the best possible condition. 

\bibliographystyle{plainnat}
\bibliography{citations}
\end{document}